\documentclass[12pt,fleqn] {article}
\usepackage{amssymb,amsmath,amsthm,setspace}
\usepackage{graphicx}
\newcommand{\R}{\mathbb{R}}

\newcommand{\adb}{\allowdisplaybreaks}

\newtheorem{Theorem}{Theorem}

\newtheorem{Corollary}{Corollary}

\newcommand{\inv}{\frac{1}}

\usepackage{multicol}
\usepackage{amssymb}
\usepackage{natbib}

\usepackage{color}

\date{\today}
\title{Quantile function expansion using regularly varying functions}
\author{Thomas Fung$^{a,}\footnote{Corresponding Author. Honorary Associate, University of Sydney. Email address:  {\tt thomas.fung@mq.edu.au} (Thomas Fung). }$\,\, and Eugene Seneta$^{b}$ \\
{\small $^a$ Department of Statistics, Macquarie University, NSW 2109, Australia}\\
{\small $^b$ School of Mathematics and Statistics, University of Sydney, NSW 2006, Australia}\\
}

\begin{document}
\maketitle

\begin{abstract}
We present a simple result that allows us to evaluate the asymptotic order of the remainder of a partial asymptotic expansion of the quantile function $h(u)$ as $u\to 0^+$ or $1^-$. This is focussed on important univariate distributions when $h(\cdot)$ has no simple closed form, with a view to assessing asymptotic rate of decay to zero of tail dependence in the context of bivariate copulas. The Introduction motivates the study in terms of the standard Normal. The Normal, Skew-Normal and Gamma are  used as initial examples. Finally, we discuss  approximation to the lower quantile of the Variance-Gamma and Skew-Slash distributions.
 
\end{abstract}

\section{Introduction}
\label{intro}

This paper is motivated by the need for a generally applicable procedure to study the asymptotic behaviour as $u\to 0^+$ of $F_i^{-1}(u)$, $i=1,2$ and $C(u,u) = P(X_1\leq F_1^{-1}(u), X_2 \leq F_2^{-1}(u))$ where the $F_i^{-1}(u)$'s are the inverse of continuous and strictly increasing cdf's $F_i(u)$, $i=1,2$, and a bivariate copula function respectively.

A random vector $\textbf{X} = (X_1, X_2)^{\top}$ with marginal inverse distribution function $F_i^{-1}(u)$, $i=1,2$ has coefficient of lower tail dependence $\lambda_L$ if the limit $\lambda_L = \lim_{u\to 0^+} \lambda_L(u)$ exists, where
\begin{equation*}
\lambda_L(u)  = P\left(X_1\leq F_1^{-1}(u) | X_2\leq F_2^{-1}(u)\right)  = \frac{P\left(X_1\leq F_1^{-1}(u), X_2\leq F_2^{-1}(u)\right)}{P\left(X_2\leq F_2^{-1}(u)\right)} = \frac{C(u ,u)}{u}. 
\end{equation*}
If $\lambda_L=0$ then $\textbf{X}$ is said to be asymptotically independent in the lower tail. In this situation in particular the asymptotic rate of approach to the limit 0 of $\lambda_L(u)$ is an indication of the strength of asymptotic independence. 

The classical case is the bivariate Normal with correlation coefficient $\rho$ as discussed in \cite{EMS2001}. It was shown in \cite{FS2011} that if 
\begin{equation}
\lambda(u) = 2\Phi\left(\Phi^{-1}(u)\sqrt{\frac{1-\rho}{1+\rho}}\right) \sim u^{\theta}L(u)\label{lower tail dependence Normal}
\end{equation}
where $\Phi(x)$, $-\infty<x<\infty$ is the cdf of the standard Normal, and $L(u)$ is a slowly varying function as $u\to 0^+$, then 
\begin{equation*}
\lambda_L(u) \sim \frac{u^{\theta}L(u)}{\theta+1}. 
\end{equation*}
\cite{FS2011} also showed in their Theorem 3 that 
\begin{align}
\lambda(u) &\sim u^{\frac{1-\rho}{1+\rho}}L(u) \quad \text{where} \label{normal slowly varying function}\\
\notag L(u) & \sim 2\sqrt{\frac{1+\rho}{1-\rho}}\left(-4\pi \log u\right)^{-\frac{\rho}{1+\rho}}. 
\end{align}
The proof within their Theorem 2 depended heavily on the very specific asymptotic relation as $x\to -\infty$ between the cdf $\Phi$ and the corresponding standard Normal pdf $f$.

We were unable to use in that paper, for this purpose, the expression for the quantile function $\Phi^{-1}(u)$ of the standard Normal distribution 
\begin{equation}
y(u) = -\sqrt{-2\log(u\sqrt{-4\pi \log u})},  \label{eqnA}
\end{equation}
given for example by \cite{LT1997} from a truncated expansion of $\Phi^{-1}(u)$, since we did not know the asymptotic order of the reminder $\Phi^{-1}(u) - y(u)$. We were able to prove that $\Phi^{-1}(u) \sim y(u)$ as $u\to 0^+$, but inasmuch as this asserts only that 
\begin{equation*}
\Phi^{-1}(u) = y(u)\left(1+o(1)\right) = -\sqrt{-2\log u}\left(1+o(1)\right),
\end{equation*}
we can only be sure of the dominant term of a truncated asymptotic expansion, and this was inadequate to proceed from (\ref{lower tail dependence Normal}). 

However our general result below, when applied to the standard Normal, gives 
\begin{equation*}
\Phi^{-1}(u) = y(u) \left(1+O\left(\frac{\log|\log u|}{(\log u)^2}\right)\right).
\end{equation*}
Since 
\begin{equation*}
\Phi(y) \sim - y^{-1}\inv{\sqrt{2\pi}}e^{-\frac{y^2}{2}}
\end{equation*}
as $y\to -\infty$, from (\ref{lower tail dependence Normal}) 
\begin{equation*}
\lambda(u) \sim k(u) e^{-\inv{2}\left(y(u)\sqrt{\frac{1-\rho}{1+\rho}}\left(1+O\left(\frac{\log|\log u|}{(\log u)^2}\right)\right)\right)^2}
\end{equation*}
where $k(u) = \frac{2}{\sqrt{-2\log u}}\inv{\sqrt{2\pi}} \sqrt{\frac{1+\rho}{1-\rho}}$;{\adb
\begin{align*}
 & = k(u)e^{-\inv{2}\left(\frac{1-\rho}{1+\rho}\right)y^2(u)\left(1+O\left(\frac{\log|\log u|}{(\log u)^2}\right)\right)^2}\\
 &= k(u)e^{-\inv{2}\left(\frac{1-\rho}{1+\rho}\right)y^2(u)\left(1+O\left(\frac{\log|\log u|}{(\log u)^2}\right)\right)}\\
 & =k(u) e^{-\inv{2}\left(\frac{1-\rho}{1+\rho}\right)\left(y^2(u)+o(1)\right)}
\end{align*}}
as $u\to 0^+$, since $y^2(u)O\left(\frac{\log|\log u|}{(\log u)^2}\right) = O\left(\frac{\log|\log u|}{|\log u|}\right)$, since $y^2(u)\sim -2\log u$ as $u\to 0^+$. Thus 
\begin{equation*}
\lambda(u) \sim k(u)e^{-\inv{2}\left(\frac{1+\rho}{1-\rho}\right)y^2(u)}
\end{equation*}
and the right hand side simplifies to the right hand side of (\ref{normal slowly varying function}). 

In this note we present a simple result that allows us to evaluate the asymptotic order of the difference between $y(u)$ and $h(u)$ as $u\to 0^+$ or $1^-$, where $h(u)$ is the quantile function corresponding to a cumulative distribution function $g(\cdot)$ for important distributions where $h(u)$ has no closed form, and $y(u)$ is an asymptotic closed form expression. This is the general result which we discuss in Section 2. 
In Section 3, we will illustrate our results by considering the quantile function for the Generalised Gamma-type tail which has various commonly used distributions such as Normal, Skew-Normal, Gamma, Variance-Gamma and a Skew-Slash as special cases. Detailed extreme value structure of such distributions is important in a financial mathematics context. These individual examples will be discussed in Section 4. 

\section{Main Result}

Our main result is summarised into the following Theorem.
\begin{Theorem} \label{Thm: inverse}
Suppose that $g(x)$ is a strictly positive continuous and strictly increasing cumulative distribution function (cdf) on $(-\infty, A]$, $A<0$. Suppose further that some function $y(u) \to -\infty$ as $u\to 0^+$ satisfies 
\begin{equation}
y(g(x)) = x\left(1+O\left(\zeta(x)\right)\right) \label{Thm: inverse: cdf pdf related}
\end{equation}
as $x\to -\infty$, such that $\zeta(x)  \to  0, x \to -\infty.$  Suppose finally that $\zeta(x) = \psi\left(-\inv{x}\right)$ with 
\begin{equation*}
\psi(w) = w^{\rho} L(w) 
\end{equation*}
for some constant $\rho \geq0$ and function $L(w)$, $w>0$, slowly varying at 0. 
If $h(u)$, $u\in (0, g(A)]$ is the inverse function of $g(\cdot)$, then 
\begin{equation*}
h(u) = y(u) \left(1+O(\zeta(y(u))\right). 
\end{equation*}
\end{Theorem}
\begin{proof}
We begin with the fact that 
\begin{align*}
 g(h(u)) &= u \\
 \Rightarrow \quad y(g(h(u))) & = y(u) \\
 \Rightarrow \quad h(u)(1+O(\zeta(h(u)))) & = y(u), \quad \text{using (\ref{Thm: inverse: cdf pdf related})}
\end{align*}
Thus $\frac{h(u)}{y(u)} \to 1$ as $u\to 0^+$ and hence
\begin{equation}
\frac{\zeta(h(u))}{\zeta(y(u))} = \frac{\psi\left(-1/h(u)\right)}{\psi\left(-1/y(u)\right)} \to 1 
\label{Thm: correction term equivalency}
\end{equation}
by the Uniform Convergence Theorem of slowly varying function of \citet[Theorem 1.1]{S1976}. 

From (\ref{Thm: correction term equivalency}), $\frac{\zeta(h(u))}{\zeta(y(u))} \to 1 \Rightarrow O\left(\zeta(h(u))\right) = O\left(\zeta(y(u)\right)$ as $u\to 0^+$. 
Finally, 
\begin{align*}
& h(u)(1+O(\zeta(h(u))))  = y(u)\\
\Rightarrow \quad & h(u) = y(u)\left(1+O(\zeta(h(u)))\right)^{-1} \\
\Rightarrow \quad & h(u) = y(u)\left(1+O(\zeta(h(u)))\right) \\
\Rightarrow \quad & h(u) = y(u)\left(1+O(\zeta(y(u)))\right). 
\end{align*}
\end{proof}
 Note  that the formulation of the theorem requires, in the case $\rho=0$ only,  that the function $L(w) \to 0, w \to 0.$ 

The condition required in Theorem \ref{Thm: inverse} is simply that the correction term in (\ref{Thm: inverse: cdf pdf related}) is related to a regular varying function which is quite general and should apply to a wide range of distributions. The result in Theorem \ref{Thm: inverse} not only ensures $h(u)$ and $y(u)$ will be asymptotically equivalent, it will also stipulate how accurate $y(u)$ will be for $h(u)$.

We express as a corollary to Theorem \ref{Thm: inverse} the corresponding result for the upper tail quantiles. 
\begin{Corollary}
\label{Cor: inverse upper}
Suppose that $g(x)$ is a strictly positive continuous and strictly increasing cumulative distribution function (cdf) on $[A, \infty)$, $A>0$. Suppose further that some function $y(u) \to \infty$ as $u\to 1^-$ satisfies 
\begin{equation*}
y(g(x)) = x\left(1+O\left(\zeta(x)\right)\right) 
\end{equation*}
as $x\to \infty$, such that $\zeta(x)  \to  0, x \to \infty.$  If  $\zeta(x) = \psi\left(\inv{x}\right)$ with 
\begin{equation*}
\psi(w) = w^{\rho} L(w) 
\end{equation*}
for some constant $\rho \geq 0$ and function $L(w)$, $w>0$, slowly varying at 0. 
If $h(u)= g^{-1}(u), u\in [g(A), 1)$ is the inverse function of $g(\cdot)$ i.e. the upper tail quantile function then 
\begin{equation*}
h(u) = y(u) \left(1+O(\zeta(y(u))\right). \label{Cor: inverse: inverse accuracy upper}
\end{equation*}
\end{Corollary}
In the next section, we will illustrate our results by considering the quantile function corresponding to a cdf that has a Generalised Gamma-type tail behaviour.

\section{Quantile for Generalised Gamma-type tail behaviour}
\subsection{Lower Tail}

We are interested in approximation to the quantile function for the Generalised Gamma-type tail.  We consider a cdf $g$ to have a Generalised Gamma-type (lower) tail behaviour if $g$ can be expressed as 
\begin{equation}
g(x) = a|x|^b e^{-c|x|^d}\left(1+O\left(\inv{|x|^e}\right)\right), \quad x<0, \label{GG tail behaviour}
\end{equation}
for some constants $a$, $c$, $d$, $e>0$ and $b\in \R$. 
Several   distributions of current interest  have such tails, and we discuss them as  special cases in the next section.

Suppose that an approximation to the quantile function $h(u) = g^{-1}(u) $  to  be 
\begin{eqnarray}
y(u) = -\left\{ \frac{-b}{cd}\left[\log\left(\frac{ \frac{cd}{|b|} \left(\frac{u}{a}\right)^{\frac{d}{b}}}{\left| \log\frac{cd}{|b|}  u^{\frac{d}{b}}\right|}\right)\right] \right\}^{\inv{d}}, \quad \text{for  small $u>0$.}  \label{GG quantile approximation}
\end{eqnarray}
This  can be obtained via the recursive method for an inverse function (see Chapter 2.4 of \cite{D1961} for instance) on $g(\cdot)$. 
Then {\adb
\begin{eqnarray*}
-y(g(x)) &=& \left\{ \frac{-b}{cd}\left[\log\left(\frac{\frac{cd}{|b|}\left(\frac{a|x|^b e^{-c|x|^d}\left(1+O\left(\inv{|x|^e}\right)\right)}{a}\right)^{\frac{d}{b}}}{\left|\log\frac{cd}{|b|}\left(a|x|^b e^{-c|x|^d}\left(1+O\left(\inv{|x|^e}\right)\right)\right)^{\frac{d}{b}}\right|}\right)\right]\right\}^{\inv{d}}\\
&=& \left\{ \frac{-b}{cd}\left[ \log\left(\frac{\frac{cd}{|b|}|x|^d e^{-\frac{cd}{b}|x|^d}\left(1+O\left(\inv{|x|^e}\right)\right)}{\left|\log\frac{cda^{\frac{d}{b}}}{|b|}|x|^d e^{-\frac{cd}{b}|x|^d}\left(1+O\left(\inv{|x|^e}\right)\right)\right|}\right)\right]\right\}^{\inv{d}}\\
&=& \Biggl\{ \frac{-b}{cd}\Biggl[ - \frac{cd}{b}|x|^d+
\log\left(\frac{cd}{|b|}|x|^d\right)+\log\left(1+O\left(\inv{|x|^e}\right)\right)\\
&& \quad -\log\left|- \frac{cd}{b}|x|^d+\log\left(\frac{cda^{\frac{d}{b}}}{|b|}|x|^d\right)+ \log\left(1+O\left(\inv{|x|^e}\right)\right)\right|\Biggr]\Biggr\}^{\inv{d}}\\
&=& \Biggl\{ -\frac{b}{cd}\Biggl[- \frac{cd}{b}|x|^d+ \log\left(\frac{cd}{|b|}|x|^d\right) +O\left(\inv{|x|^e}\right) -\log\left(\frac{cd}{|b|}|x|^d\right)\\
&& \quad - \log\left(1- \frac{\log\left(\frac{cda^{\frac{d}{b}}}{|b|}|x|^d\right)+O\left(\inv{|x|^e}\right)}{\frac{cd}{b}|x|^d}\right)\Biggr]\Biggr\}^{\inv{d}}\\
&=& \left\{\frac{-b}{cd}\left[-\frac{cd}{b}|x|^d + O\left(\frac{\log |x|}{|x|^d}\right) + O\left(\inv{|x|^{e}}\right)\right]\right\}^{\inv{d}}\\
&=& |x|\left(1+O\left(\inv{|x|^{e+d}}\right)+O\left(\frac{\log|x|}{|x|^{2d}}\right)\right)\\
&=& 
\begin{cases}
|x|\left(1+ O\left(\frac{\log|x|}{|x|^{2d}}\right)\right), & \text{if $d\leq e$} \\
|x|\left(1+O\left(\inv{|x|^{d+e}}\right)\right), &\text{otherwise.}
\end{cases}
\end{eqnarray*}}
As $y(u)\sim - \left(-\inv{c}\log u\right)^{\inv{d}})$, the behaviour of the quantile function $h(\cdot)$ is
\begin{equation}
h(u) = \begin{cases}
y(u)\left(1+O\left(\frac{\log\left|\left(-\inv{c}\log u\right)^{\inv{d}}\right|}{(\log u)^2}\right)\right) = y(u)\left(1+O\left(\frac{\log\left|\log u\right|}{(\log u)^2}\right)\right), & \text{if $d\leq e$}\\
y(u)\left(1+O\left(\inv{|\log u|^{\frac{e}{d}+1}}\right)\right), & \text{otherwise},
\end{cases} \label{h general}
\end{equation}
by Theorem \ref{Thm: inverse}.

We next expand $y(u)$ to obtain an expansion for $h(u)$ with an error term of appropriate lower order. After some algebra, {\adb
\begin{eqnarray*}
 y(u) 
&=&  -\left(-\inv{c}\log u \right)^{\inv{d}}\Biggl[1- \frac{b\log|\log u|}{d^2\log u} - \frac{b\log\left(\frac{a^{\frac{d}{b}}}{c}\right)}{d^2\log u} + \frac{b^2}{2d^3}\left(\inv{d}-1\right)\frac{(\log|\log u|)^2}{(\log u)^2} \\
&& \quad +\frac{b^2}{d^3}\left(\inv{d}-1\right)\log\left(\frac{a^{\frac{d}{b}}}{c}\right)\frac{\log|\log u|}{(\log u)^2}+ \frac{\frac{b^2}{d^3} \left[\left(\inv{2d}-\inv{2}\right) \left(\log\left(\frac{a^{\frac{d}{b}}}{c}\right)\right)^2- \log\left(\frac{cd}{|b|}\right)\right]}{(\log u)^2}\\
&& \quad + O\left(\frac{(\log|\log u|)^3}{(|\log u|)^3}\right) \Biggr].
\end{eqnarray*}}
As a result, when $d\leq e$, $h(\cdot)$ becomes {\adb
\begin{align}
\notag  h(u) = &   y(u)\left(1+O\left(\frac{\log\left|\log u\right|}{(\log u)^2}\right)\right) \\
\notag = &-\left(-\inv{c}\log u\right)^{\inv{d}} - \frac{b\log|\log u|}{cd^2\left(-\inv{c}\log u\right)^{1-\inv{d}}} - \frac{b\log\left(\frac{a^{\frac{d}{b}}}{c}\right)}{cd^2(-\inv{c}\log u)^{1-\inv{d}}}\\
& \quad - \frac{b^2\left(\inv{2d}-\inv{2}\right)(\log|\log u|)^2}{c^2d^3(-\inv{c}\log u)^{2-\inv{d}}}+ O\left(\frac{\log|\log u|}{|\log u|^{2-\inv{d}}}\right); \label{h(u) expansion d <= e}
\end{align}}

On the other hand, when $d>e$, $h(\cdot)$ becomes {\adb
\begin{align*}
\notag  h(u) = &  y(u)\left(1+O\left(\inv{|\log u|^{\frac{e}{d}+1}}\right)\right) \\
\notag = & -\left(-\inv{c}\log u\right)^{\inv{d}} - \frac{b\log|\log u|}{cd^2\left(-\inv{c}\log u\right)^{1-\inv{d}}} 
- \frac{b\log\left(\frac{a^{\frac{d}{b}}}{c}\right)}{cd^2\left(-\inv{c}\log u\right)^{1-\inv{d}}}\\
& \quad + O\left(\frac{1}{|\log u|^{\frac{e}{d}+1-\inv{d}}}\right). 
\end{align*}}
\subsection{Upper Tail}

Similarly, a cdf $g(\cdot)$ is said to have a Generalised Gamma-type (upper) tail behaviour if $g$ can be expressed as 
\begin{equation}
g(x) = 1- ax^b e^{-cx^d}\left(1+O\left(\inv{x^e}\right)\right), \quad x \to \infty, \label{GG tail behaviour upper}
\end{equation}
for some constants $a$, $c$, $d$, $e>0$ and $b\in \R$, which would  suggest the following approximation to the upper tail  quantile function: 
\begin{equation}
y(u) =  \left\{ \frac{-b}{cd}\left[\log\left(\frac{ \frac{cd}{|b|} \left(\frac{(1-u)}{a}\right)^{\frac{d}{b}}}{\left| \log\frac{cd}{|b|} (1-u)^{\frac{d}{b}}\right|}\right)\right] \right\}^{\inv{d}}, \quad \text{as  $u\to 1^{-}$.}  \label{GG quantile approximation upper}
\end{equation}
Then by Corollary \ref{Cor: inverse upper}, the upper tail quantile function $h(\cdot)$ can be expressed as {\adb
\begin{align}
\notag  & h(u) \\
\notag  =& \begin{cases}
 y(u)\left(1+O\left(\frac{\log\left|\log (1-u)\right|}{(\log(1- u))^2}\right)\right), & \text{if $d\leq e$;}\\
y(u)\left(1+O\left(\inv{|\log (1-u)|^{\frac{e}{d}+1}}\right)\right), & \text{if $d<e$};
\end{cases} \\
=& 
\begin{cases}
\left(-\inv{c}\log(1-u)\right)^{\inv{d}} + \frac{b\log|\log(1-u)|}{cd^2\left(-\inv{c}\log(1-u)\right)^{1-\inv{d}}} + \frac{b\log\left(\frac{a^{\frac{d}{b}}}{c}\right)}{cd^2(-\inv{c}\log(1-u))^{1-\inv{d}}}\\
\quad + \frac{b^2\left(\inv{2d}-\inv{2}\right)(\log|\log(1-u)|)^2}{c^2d^3(-\inv{c}\log(1-u))^{2-\inv{d}}}+ O\left(\frac{\log|\log(1-u)|}{|\log(1-u)|^{2-\inv{d}}}\right), & \text{if $d\leq e$};\\
\left(-\inv{c}\log(1-u)\right)^{\inv{d}} + \frac{b\log|\log(1-u)|}{cd^2\left(-\inv{c}\log(1-u)\right)^{1-\inv{d}}} 
+\frac{b\log\left(\frac{a^{\frac{d}{b}}}{c}\right)}{cd^2\left(-\inv{c}\log(1-u)\right)^{1-\inv{d}}}\\
\quad + O\left(\frac{1}{|\log (1-u)|^{\frac{e}{d}+1-\inv{d}}}\right), & \text{if $d>e$}.
\end{cases}
\label{h general upper}
\end{align}}
as $u \to 1^-$.

\section{Applications}

In this section, we will illustrate our results with  application to  the Normal, Skew-Normal, Gamma, Variance-Gamma and Skew-Slash distributions.
\subsection{ Standard Normal}

From \cite{F1968} Chapter VII Lemma 2, the form of $g(\cdot)$ is 
\begin{equation*}
g(x) = \inv{\sqrt{2\pi}|x|}e^{-\frac{x^2}{2}}\left(1+O\left(\inv{x^2}\right)\right), \quad \text{for $x<1$}.
\end{equation*}
If we compare this with (\ref{GG tail behaviour}), we have $a= \inv{\sqrt{2\pi}}$, $b = -1$, $c = \inv{2}$ and  $d= e = 2$. This means that the approximation $y(\cdot)$ becomes: 
\begin{equation}
y(u) = -\left\{ \log\left(\frac{(\sqrt{2\pi} u)^{-2}}{\left| \log(u) ^{-2}\right|}\right)\right\}^{\inv{2}} = - \sqrt{-2\log \left(u\sqrt{4\pi|\log u|}\right)} \label{approx normal quantile}
\end{equation}
and is the same as (\ref{eqnA}). As in this case we have $d=e$, we can obtain the behaviour of the quantile function by using using (\ref{h general}) and (\ref{h(u) expansion d <= e}) and get 
\begin{align*}
h(u) = & y(u)\left(1+O\left(\frac{\log|\log u|}{\left(\log u\right)^2}\right)\right) \\
= & - \sqrt{-2\log u} + \frac{\log\left|\log u\right|}{2\sqrt{-2\log u}} + \frac{\log 4\pi}{2\sqrt{-2\log u}}\\
& + \frac{\left(\log|\log u|\right)^2}{8\left(-2\log u\right)^{\frac{3}{2}}} + O\left(\frac{\log|\log u|}{|\log u|^{\frac{3}{2}}}\right).
\end{align*}
If  we use only the dominating term of $y(u)$ in (\ref{approx normal quantile}) and so  put 
\begin{equation*}
y^*(u) = -\sqrt{-2\log u}
\end{equation*}
then after some algebra we have 
%
\begin{equation*}
y(g(x)) = x(1+O(\zeta(x)))\quad \text{where} \quad \zeta(x) = \frac{\log |x|}{x^2}. 
\end{equation*}
As a result, 
\begin{equation*}
h(u) = y^*(u)\left(1+ O\left(\frac{\log|\log u|}{\log u}\right)\right)
\end{equation*}
by using Theorem \ref{Thm: inverse}.  We can see that $y^*(u)$ is a less accurate approximation to $h(u)$ than $y(u),$ and the expression
for $h(u)$ that it gives only improves on   $y^*(u)$ by giving the correct order of the difference $h(u) -  y^*(u)$.

We are aware that there exists a whole field of literature on finding an efficient and accurate approximation in the numerical sense for the Normal quantile functions, such as \cite{AS1964}, \cite{BS1977} to the more recent \cite{V2010}. \cite{SE2014} provide a substantial bibliography on this subject. But that is not our focus and so our methodology will most likely be outperformed by the more sophisticated approximation in the literature. Take the approximation to the Normal quantile in Section 2.2.1 in \cite{V2010} for example which gives
\begin{equation*}
y_V(u) = c_3\sqrt{-2\log u} + c_2' + \frac{c_1'\sqrt{-2\log u}+c_0'}{-2\log u+d_1\sqrt{-2\log u}+d_0}
\end{equation*}
for $e^{-37^2/2} <u<  0.0465$ where 
\begin{align*}
c_3 &=-1.000182518730158122,\\
c_0' &= 16.682320830719986527, \\
c_1' &= 4.120411523939115059, \\
c_2' &= 0.029814187308200211,\\ 
d_0 &= 7.173787663925508066, \\
d_1 &= 8.759693508958633869.
\end{align*}
To see how the expansions perform against the approximations, we plot the standard Normal quantile function, denoted as $\Phi^{-1}(u)$, in {\tt R} against $y_V(\cdot)$, $y(\cdot)$ and $y^{*}(\cdot)$ and the results are shown in Fig. \ref{quantile comp}. 

\begin{figure}[h]
\begin{center}
\scalebox{0.5}{\includegraphics{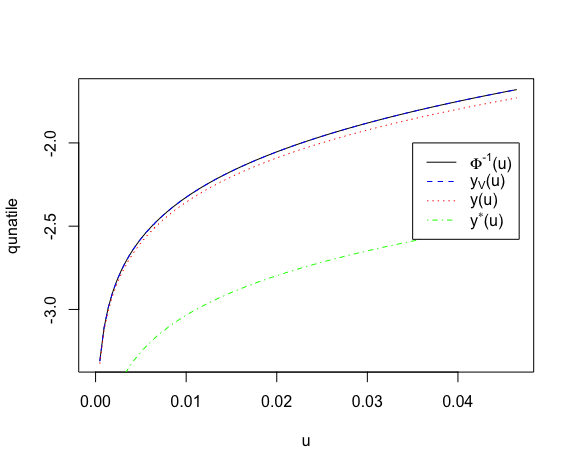}}
\caption{Comparison in {\tt R} of $\Phi^{-1}(u)$, $y_V(\cdot)$, $y(\cdot)$ and $y^*(\cdot)$}\label{quantile comp}
\end{center}
\end{figure}

From Fig. \ref{quantile comp}, we can see that $h_V(\cdot)$ gives almost identical result to the built-in Normal quantile function in {\tt R} and is better than our expansion $y(u)$. We also see that $y^*(\cdot)$ is an ``order'' worse than the other methods.

\subsection{Skew-Normal}
 The quantile function of the Skew-Normal distribution is another example that is covered by our result. The distribution was first introduced by \cite{A1985} and has developed into an extensive theory presented in a recent monograph by  \cite{AC2014}. A random variable is said to have a standard Skew-Normal distribution if its density function is
\begin{equation*}
f(x) = \frac{2}{\sqrt{2\pi}}e^{-\inv{2}x^2}\Phi(\lambda x), \quad x\in \R
\end{equation*}
where $\lambda\in R$ controls the skewness of the distribution. When $\lambda=0$, it reduces to the standard Normal as special case so we shall exclude the case of $\lambda=0$ from our subsequent discussion. Using Lemma 2 of  \cite{C2010} or \citet[pp. 52--53]{AC2014}, we have 
\begin{equation}
g(x)  = 
\begin{cases}
\inv{\pi\lambda(1+\lambda^2)|x|^2}e^{-\inv{2}(1+\lambda^2)x^2}\left(1+O\left(\inv{x^2}\right)\right), & \lambda>0;\\
\frac{2}{\sqrt{2\pi}|x|}e^{-\inv{2}x^2}\left(1+O\left(\inv{x^2}\right)\right), &  \lambda<0,
\end{cases} \quad \text{$x\to -\infty$}.
\label{skew normal g}
\end{equation}
This means if we compare (\ref{skew normal g}) with (\ref{GG tail behaviour}), we have 
$a = \inv{\pi\lambda(1+\lambda^2)}$, $b = -2$, $c = (1+\lambda^2)/2$ and $d = e = 2$ when $\lambda>0$; $a = 2/\sqrt{2\pi}$, $b = -1$, $c= 1/2$, $d = e = 2$ when $\lambda<0$.  This means that the approximation $y(\cdot)$ becomes:
\begin{align}
\notag y(u) =&  \begin{cases} 
- \left\{ \frac{2}{(1+\lambda^2)}\left[\log\left(\frac{\frac{(1+\lambda^2)}{2}\left(u \pi\lambda(1+\lambda^2)\right)^{-1}}{\left|\log\frac{(1+\lambda^2)}{2}\left(u\right)^{-1}\right|}\right)\right]\right\}^{\inv{2}},  & \lambda>0;\\
- \left\{ \log\left(\frac{(u\sqrt{2\pi}/2)^{-2}}{\left|\log\left(u\right)^{-2}\right|}\right)\right\}^{\inv{2}}, & \lambda<0;
\end{cases}\\
= & \begin{cases}
 -\sqrt{ - \frac{2}{1+\lambda^2}\log(- 2\pi\lambda u \log(2u/(1+\lambda^2)))}
, & \text{if $\lambda>0$;}\\
 -\sqrt{-2\log\left(\frac{u}{2}\sqrt{-4\pi \log u})\right)}
, & \text{if $\lambda<0$.}
\end{cases} \label{approx skew normal quantile}
\end{align}

As in this case we have $d = e$, we can obtain the behaviour of the quantile function by using (\ref{h general}) and (\ref{h(u) expansion d <= e}) to get {\adb

\begin{align}
h(u) = & y(u) \left(1+O\left(\frac{\log|\log u|}{(\log u)^2}\right)\right)  \label{approx skew normal quantile 2}\\
\notag = & \begin{cases}
-\sqrt{-\frac{2}{1+\lambda^2}\log u} + \frac{\log \left| \log u\right|}{(1+\lambda^2)\sqrt{-\frac{2}{1+\lambda^2}\log u}} + \frac{\log(2\pi\lambda)}{(1+\lambda^2)\sqrt{-\frac{2}{1+\lambda^2}\log u}}\\
\quad + \frac{\left(\log|\log u|\right)^2}{2(1+\lambda^2)^2\left(-\frac{2}{1+\lambda^2}\log u\right)^{\frac{3}{2}}}+O\left(\frac{\log|\log u|}{|\log u|^{\frac{3}{2}}}\right), & \lambda>0; \\
- \sqrt{-2\log u } + \frac{\log\left|\log u\right|}{2\sqrt{-2\log u}} + \frac{\log \pi}{2\sqrt{-2\log u}} \\
\quad + \frac{\left(\log|\log u|\right)^2}{8(-2\log u)^{\frac{3}{2}}} +O\left(\frac{\log|\log u|}{|\log u|^{\frac{3}{2}}}\right), & \lambda<0.
\end{cases}
\end{align}}

The expression (\ref{approx skew normal quantile 2}) justifies the use of expression (\ref{approx skew normal quantile}) above in Theorem 2 of \cite{FS2016}.


\subsection{Gamma} 

 Using (\ref{h general upper}), one can also determine the accuracy of the gamma-like upper tail. Suppose that $X\sim \Gamma(\alpha, \beta)$ and let its pdf be 
\begin{equation*}
f(x) = \frac{\beta^{\alpha}}{\Gamma(\alpha)}x^{\alpha-1}e^{-\beta x}, \quad x>0.
\end{equation*} 
Then 
\begin{eqnarray}
g(x) = 1- \frac{\beta^{\alpha-1}}{\Gamma(\alpha)} x^{\alpha-1}e^{-\beta x}\left(1+O\left(\inv{x}\right)\right),\label{gamma g}
\end{eqnarray}
as $x\to \infty$. This means that if we compare (\ref{gamma g}) with (\ref{GG tail behaviour upper}), we have 
$a = \frac{\beta^{\alpha-1}}{\Gamma(\alpha)}$, $b = \alpha-1$, $c = \beta$ and $d = e =1$. By using (\ref{GG quantile approximation upper}) an approximation to the upper quantile function when $u\to 1^{-}$ would be 
\begin{align*}
y(u) =& \left\{\frac{1-\alpha}{\beta}\left[\log\left(\frac{\frac{\beta}{|\alpha-1|}\left(\frac{1-u}{\beta^{\alpha-1}/\Gamma(\alpha)}\right)^{\inv{\alpha-1}}}{\left|\log \left(\frac{\beta}{|\alpha-1|}\left(1-u\right)^{\inv{\alpha-1}}\right)\right|}\right)\right]\right\}\\
=& -\inv{\beta}\log\left(\frac{(1-u)\Gamma(\alpha)}{\left|\log\left(\frac{(1-u) \beta^{\alpha-1}}{|\alpha-1|^{\alpha-1}}\right)\right|^{\alpha-1}}\right)
\end{align*}
As in this case we have $d = e$ again, we have for the upper quantile function
\begin{eqnarray*}
h(u) &=& y(u)\left(1+O\left(\frac{\log|\log(1-u)|}{(\log(1-u))^2}\right)\right) \\
&=&  -\inv{\beta}\log\left(\frac{(1-u)\Gamma(\alpha)}{\left|\log\left(\frac{(1-u) \beta^{\alpha-1}}{|\alpha-1|^{\alpha-1}}\right)\right|^{\alpha-1}}\right)\left(1+O\left(\frac{\log|\log(1-u)|}{(\log(1-u))^2}\right)\right) \\
&=& -\inv{\beta}\log\left(1-u\right) + \frac{\alpha-1}{\beta}\log\left|\log(1-u)\right| - \inv{\beta}\log\Gamma(\alpha)+ O\left(\frac{\log|\log(1- u)|}{|\log(1-u)|}\right)
\end{eqnarray*}
as $u\to 1^{-}$ by using (\ref{h general upper}). 

In  Section 4 of  \cite{FS2011}, via their Theorem 1, the authors obtained $h(u) = y(u)( 1 + o(1)),  \, u \to 1^{-}$  with $ y(u) =  -\inv{\beta}\log\left(1-u\right). $

\subsection{Variance-Gamma and Skew-Slash}

 Finally, we  propose an approximation to the lower tail quantile function of the Variance-Gamma (VG) and a Skew-Slash distribution, as they can have similar tail structure. A random variable $X$ is said to have a skew VG distribution introduced in \cite{MCC1998} and further studied in  \cite{S2003}, \cite{S2004} and  \cite{TS2006}, if X is defined by a Normal variance-mean mixture as $X|Y \sim N(\mu+\theta Y, \sigma^2 Y)$, where $\mu, \theta \in \R$ and $\sigma^2>0$. The distribution of $Y$ is $\Gamma\left(\inv{\nu}, \inv{\nu}\right)$ with $\nu>0$ such that $EY = 1$. When $\mu=0$ and $\sigma =1$, from \cite{S2003}, we have 
%
%
\begin{eqnarray}
\notag g(x) &=& \frac{|x|^{\inv{\nu}-1}e^{-\left(\sqrt{\frac{2}{\nu}+\theta^2}+\theta\right)|x|}}{\nu^{\inv{\nu}}\Gamma\left(\inv{\nu}\right)\left(\frac{2}{\nu}+\theta^2\right)^{\inv{2\nu}}\left(\sqrt{\frac{2}{\nu}+\theta^2}+\theta\right)}\left(1+O\left(\inv{|x|}\right)\right) \\
&=& a |x|^{\inv{\nu}-1}e^{-\left(\sqrt{\frac{2}{\nu}+\theta^2}+\theta\right)|x|}\left(1+O\left(\inv{|x|}\right)\right) \label{VG tail}
\end{eqnarray}
as $x\to -\infty$, where $a>0$ is  defined by the above.. By comparing with (\ref{GG tail behaviour}), $a$ is defined as above, $b = \inv{\nu}-1$, $c = \sqrt{\frac{2}{\nu}+\theta^2}+\theta$ and $d = e=1$ which in turn would suggest the following approximation to the lower quantile function: {\adb
\begin{eqnarray*}
y(u) &=& \inv{\sqrt{\frac{2}{\nu}+\theta^2}+\theta}\log\left(\frac{u}{a}\left(\frac{\sqrt{\frac{2}{\nu}+\theta^2}+\theta}{\left|\log\left[ u\left(\frac{\sqrt{\frac{2}{\nu}+\theta^2}+\theta}{|\inv{\nu}-1|}\right)^{\inv{\nu}-1}\right]\right|}\right)^{\inv{\nu}-1}\right) \\
&\sim& \frac{\log u}{\sqrt{\frac{2}{\nu}+\theta^2}+\theta}, 
\end{eqnarray*}}
as 
$u\to 0^+$ by using (\ref{GG quantile approximation}). We can then obtain the behaviour of the quantile function as  {\adb
\begin{eqnarray*}
&& h(u) \\
&=& \inv{\sqrt{\frac{2}{\nu}+\theta^2}+\theta}\log\left(\frac{u}{a}\left(\frac{\sqrt{\frac{2}{\nu}+\theta^2}+\theta}{\left|\log\left[ u\left(\frac{\sqrt{\frac{2}{\nu}+\theta^2}+\theta}{|\inv{\nu}-1|}\right)^{\inv{\nu}-1}\right]\right|}\right)^{\inv{\nu}-1}\right)  \left(1+O\left(\frac{\log|\log u|}{(\log u)^2}\right)\right)\\
&=& \frac{\log u}{\sqrt{\frac{2}{\nu}+\theta^2}+\theta} - \frac{\left(\inv{\nu}-1\right)\log|\log u|}{\sqrt{\frac{2}{\nu}+\theta^2}+\theta} - \frac{\left(\inv{\nu}-1\right)\log\left(\frac{a^{\frac{\nu}{1-\nu}}}{\sqrt{\frac{2}{\nu}+\theta^2}+\theta}\right)}{\sqrt{\frac{2}{\nu}+\theta^2}+\theta}+ O\left(\frac{\log|\log u|}{|\log u|}\right)
\end{eqnarray*}
as $u\to 0^+$, by using (\ref{h general}) and (\ref{h(u) expansion d <= e}).
}

A Skew-Slash distribution on the other hand was first proposed in multivariate form in  \cite{A2008} and further studied in  \cite{LP2015}. A random variable $X$ is said to have a Skew-Slash distribution if $X$ is defined by a Normal variance-mean mixture as $X|Y \sim N(\mu+\theta/Y, \sigma^2Y)$, where $\mu$, $\theta\in\R$ and $\sigma^2>0$. The distribution of $Y$ is Beta$(\lambda,1$); that is, its pdf is: $f(y) = \lambda y^{\lambda - 1}, \, 0 <y <1 ; =0 \, \, \text{otherwise}.$ Here we only consider the case of $\theta>0$  so that the tail behaviour is similar to that of the Variance-Gamma in (\ref{VG tail}).
When $\mu=0$, $\sigma=1$ and $\theta>0$, from \citet[Lemma 2.1 and equation (11)]{LP2015}, we have 
\begin{equation*}
g(x) = \frac{\lambda \theta^{\lambda-1}}{2} |x|^{-(\lambda+1)} e^{-2\theta|x|}\left(1+O\left(\inv{|x|}\right)\right),
\end{equation*}
as $x\to -\infty$. By comparing this with (\ref{GG tail behaviour}), we have $a = \frac{\lambda\theta^{\lambda-1}}{2}$, $b = -(\lambda+1)$, $c= 2\theta$, and $d=e=1$, which in turn  suggests the following approximation to the lower quantile function: 
\begin{equation*}
y(u) = -\frac{(\lambda+1)}{2\theta}\left[\log\left(\frac{\frac{2\theta}{\lambda+1}\left(\frac{u}{\lambda\theta^{\lambda-1}/2}\right)^{-\inv{\lambda+1}}}{\left|\log\frac{2\theta}{\lambda+1}u^{-\inv{\lambda+1}}\right|}\right)\right] \sim \frac{\log u}{2\theta},
\end{equation*}
as $u\to 0^+$ by using (\ref{GG quantile approximation}). \citet[Lemma 3.1, equation (3)]{LP2015} use the fact that $h(u) = y(u)(1+o(1))$, $u\to 0^+$ with $y(u) = \log u/2\theta$. We can  obtain the behaviour of the quantile function as 
\begin{align*}
h(u) & =-\frac{(\lambda+1)}{2\theta}\left[\log\left(\frac{\frac{2\theta}{\lambda+1}\left(\frac{u}{\lambda\theta^{\lambda-1}/2}\right)^{-\inv{\lambda+1}}}{\left|\log\frac{2\theta}{\lambda+1}u^{-\inv{\lambda+1}}\right|}\right)\right] \left(1+O\left(\frac{\log|\log u|}{(\log u)^2}\right)\right)\\
&= \frac{\log u}{2\theta} + \frac{(\lambda+1)\log|\log u|}{2\theta} - \frac{\log\left(\lambda(2\theta)^{2\lambda}\right)}{2\theta}+O\left(\frac{\log|\log u|}{|\log u|}\right).
\end{align*}
as $u\to 0^+$ by using (\ref{h general}) and (\ref{h(u) expansion d <= e}) once again. 

\section*{Acknowledgements}
The authors thank the referee for a careful reading of the original version, and 
a list of suggestions which have resulted in a much improved paper.

%




\end{document}